\documentclass[12pt]{article}
\usepackage{amsmath,amsthm,amssymb}
\usepackage[left=3cm,right=3cm, top=2.5cm,bottom=2.5cm,bindingoffset=0cm]{geometry}

\usepackage{amscd}

\usepackage{hyperref}

\usepackage{mathtools}

\usepackage[all,cmtip]{xy}

\usepackage{MnSymbol} 

\makeatletter
\def\th@plain{%
  \thm@notefont{}
  \itshape 
}
\def\th@definition{%
  \thm@notefont{}
  \normalfont 
}
\makeatother

\newtheorem{lemma}{Lemma}[section]

\newtheorem{remark-definition}[lemma]{Remark-Definition}
\newtheorem{theorem}[lemma]{Theorem}

\newtheorem{conjecture}[lemma]{Conjecture}
\newtheorem{proposition-conjecture}[lemma]{Proposition-conjecture}

\theoremstyle{definition}

\begin{document}

\title{A note on Jordan–Kronecker invariants of semi-direct sums of $\operatorname{sl}(n)$ with a commutative ideal}
\author{I.\,K.~Kozlov\thanks{No Affiliation, Moscow, Russia. E-mail: {\tt ikozlov90@gmail.com} }
}
\date{}

\maketitle

\begin{abstract} K.\,S.~Vorushilov described Jordan--Kronecker invariants for semi-direct sums $\operatorname{sl} \ltimes \left(\mathbb{C}^n\right)^k$ if $k > n$ or if $n$ is a multiple of $k$. We describe the Jordan--Kronecker invariants in the cases $n \equiv \pm 1 \pmod k$. \end{abstract}

\section{Main result}

Jordan--Kronecker invariants of Lie algebras were introduced by A.\,V.~Bolsinov and P.~Zhang in \cite{BolsZhang}. Consider a semi-direct sum $\mathfrak{g} = \operatorname{sl}(n) \ltimes \left(\mathbb{C}^n\right)^k$, given by the standard representation.   K.\,S.~Vorushilov~\cite{Vor2} calculated the Jordan--Kronecker invariants for $k > n$ or  $n$ a multiple of $k$. We calculate them in several remaining cases.

\begin{theorem} Consider the semi-direct sum $\mathfrak{g} = \operatorname{sl}(n) \ltimes \left(\mathbb{C}^n\right)^k$, given by the standard representation. Assume that \begin{equation} \label{Eq:KCond} k < n, \qquad n = kd + r,  \end{equation} where  \begin{equation} \label{Eq:R} r =1 \qquad \text{ or } \qquad r = k - 1.\end{equation} Then the Jordan--Kronecker invariants of $\mathfrak{g}$ consist of \[\operatorname{ind} \mathfrak{g} = k\] equal Kronecker blocks. The size of each Kronecker block is \[(d+1)(n+r) - 1.\]

\end{theorem}

\begin{proof} As it was shown in \cite{Vor2}, the Jordan--Kronecker invariants for the considered Lie algebras consist only of Kronecker blocks. Following the notation of \cite{Vor2}, the invariants of coadjoint representation have the form $\operatorname{det} M_y$, where \[ M_y = \left( \begin{matrix} L, \quad Y^T L, \quad \dots, \quad (Y^T)^{d-1}L, \quad (Y^T)^{d}l_{i_1}, \dots  (Y^T)^{d}l_{i_r}\end{matrix} \right), \qquad L = \left( \begin{matrix}l_1, \dots, l_k \end{matrix} \right). \]  The $r$ right columns of $M_y$ are chosen from the $k$ columns of $\left(Y^T\right)^d L$. By \cite[Theorem 5]{Vorontsov09} a complete family of independent invariants can be chosen among these polynomials. If $r$ satisfies \eqref{Eq:R}, then there are $k$ such invariants $f_i$, each with degree \[\operatorname{deg} f_i = \frac{(d+1)(n+r)}{2}.\] Since $\operatorname{ind}\mathfrak{g} = k$, the polynomials $f_i$ are algebraically independent  and satisfy \[ \sum_{i=1}^k \deg f_i = \frac{1}{2}\left(\dim \mathfrak{g} + \operatorname{ind}\mathfrak{g}\right). \]  As $\mathfrak{g}$ is of of Kronecker type, by \cite[Theorem 9 and Remark 6]{BolsZhang}, the Jordan--Kronecker invariants of $\mathfrak{g}$ consist of $k$ Kronecker blocks with sizes $2 \operatorname{deg}f_i - 1$, as required. \end{proof}

Note that the Jordan--Kronecker invariants correspond to the generic skew-symmetric matrix pencils with fixed rank (see \cite{Dmytryshyn18}). We conjecture this holds for all cases. 

\begin{conjecture} Consider the semi-direct sums $\mathfrak{g} = \operatorname{sl}(n) \ltimes \left(\mathbb{C}^n\right)^k$, such that \eqref{Eq:KCond} holds and \[ 1 < r < k -1. \] Then the Jordan--Kronecker invariants of $\mathfrak{g}$ have no Jordan blocks and consist of \[\operatorname{ind} \mathfrak{g} = kr - r^2 +1 \] Kronecker blocks of "almost equal size" (i.e. their sizes do not differ more than by $1$). Namely, if \[ \frac{\operatorname{dim} \mathfrak{g} + \operatorname{ind} \mathfrak{g}}{2} = l \operatorname{ind} \mathfrak{g} + b, \qquad 0 \leq b < \operatorname{ind} \mathfrak{g},\] then there are 

\begin{itemize}
\item $b$ Kronecker blocks with size $2l+1$,

\item $\operatorname{ind}\mathfrak{g} - b$  Kronecker block with size $2l-1$.
    
\end{itemize}
    
\end{conjecture}

The author would like to thank A.\,V.~Bolsinov and A.\,M.~Izosimov for useful comments.

\end{document}